\documentclass[12pt]{amsart}

\usepackage[english]{babel}
\usepackage[utf8]{inputenc}

\usepackage{csquotes}
\usepackage{enumerate}
\usepackage{amsmath, amsthm, amsfonts, amssymb}
\usepackage{mathtools}
\usepackage[colorlinks=true, linkcolor=red, citecolor=green, backref=page]{hyperref}
\usepackage{color}
\usepackage{graphicx}
\usepackage[capitalise, noabbrev]{cleveref}
\usepackage{geometry}
\usepackage{marginnote}
\usepackage{tikz-cd}
\usepackage{centernot}
\usepackage{chngcntr}
\usepackage{mathrsfs}
\usepackage{placeins}

\usepackage{algorithmic}

\counterwithout{subsection}{section}

\newcommand\restr[2]{{
  \left.\kern-\nulldelimiterspace 
  #1 
  \vphantom{\big|} 
  \right|_{#2} 
  }}

\DeclareMathOperator{\Ker}{Ker}

\DeclareMathOperator{\Id}{Id}

\newcommand{\jump}{\vskip 2mm}

\numberwithin{equation}{section}

\theoremstyle{plain}
\newtheorem{theorem}{Theorem}[section]
\newtheorem{proposition}[theorem]{Proposition}
\newtheorem{lemma}[theorem]{Lemma}
\newtheorem{corollary}[theorem]{Corollary}

\theoremstyle{definition}

\newtheorem{remark}{Remark}[section]
\newtheorem{example}{Example}
\newtheorem{algorithm}{Algorithm}

\numberwithin{theorem}{subsection}
\numberwithin{equation}{subsection}

\title[An algorithm for Hodge ideals]{An algorithm for Hodge ideals}

\author[G. Blanco]{Guillem Blanco}

\thanks{The author is supported by a Postdoctoral Fellowship of the Research Foundation -- Flanders}

\address{Department of Mathematics\\ KU Leuven,
Celestijnenlaan 200B, 3001 Leuven, Belgium.}
\email{guillem.blanco@kuleuven.be}

\begin{document}

\begin{abstract}
We present an algorithm to compute the Hodge ideals \cite{MP19, MP19b} of \( \mathbb{Q} \)-divisors associated to any reduced effective divisor \( D \). The computation of the Hodge ideals is based on an algorithm to compute parts of the \( V \)-filtration of Kashiwara and Malgrange on \( \iota_{+}\mathscr{O}_X(*D) \) and the characterization \cite{MP20b} of the Hodge ideals in terms of this \( V \)-filtration. In particular, this gives a new algorithm to compute the multiplier ideals and the jumping numbers of any effective divisor.
\end{abstract}

\maketitle

\section{Introduction}

Let \( X \) be a smooth complex variety of dimension \( n \) and let \( D \) be a reduced effective divisor on \( X \). Consider \( \mathscr{O}_X(*D ) \), the sheaf of meromorphic functions with poles along the divisor \( D \). This is a coherent left \( \mathscr{D}_X \)-module that underlies the mixed Hodge module \( j_* \mathbb{Q}_U^H[n] \), where \( U := X \setminus D \), \( j : U \hookrightarrow X \) is the open inclusion, and \( \mathbb{Q}_U^H[n] \) is the constant pure Hodge module on \( U \), see \cite{Sai88, Sai90}. Any \( \mathscr{D}_X \)-module associated to a mixed Hodge module carries a good filtration \( F_\bullet \), the Hodge filtration of the mixed Hodge module.

\jump

It is shown in \cite{Sai93} that the Hodge filtration on \( \mathscr{O}_X(*D) \) is contained in the pole order filtration, namely
\[ F_k \mathscr{O}_X(*D) \subseteq \mathscr{O}_X((k+1)D) \quad \textnormal{for all} \quad k \geq 0. \]
In order to study the Hodge filtration on \( \mathscr{O}_{X}(*D) \), Musta\c{t}\v{a} and Popa \cite{MP19} introduced a set of ideal sheaves, the Hodge ideals \( I_k(D)\) of \( D \), that are defined by
\[ F_k \mathscr{O}_X(*D) = I_k(D) \otimes_{\mathscr{O}_X} \mathscr{O}_X((k+1)D), \quad \textnormal{for all} \quad k \geq 0. \]

In a subsequent work, Musta\c{t}\v{a} and Popa \cite{MP19b} generalized the notion of Hodge ideals to arbitrary \( \mathbb{Q} \)-divisors. If \( f \in \mathscr{O}_X(X) \) is a global regular function, denote \( D := \textnormal{div}(f) \) and let \( Z \) be the support of \( D \). Then, for \( \alpha \in \mathbb{Q}_{>0} \) one can associate to this data a twisted version of the localization \( \mathscr{D}_X \)-module considered above, namely
\[ \mathcal{M}(f^{-\alpha}) := \mathscr{O}_X(*Z) f^{-\alpha}, \]
that is, a rank one free \( \mathscr{O}_X(*Z) \)-module with generator the formal symbol \( f^{-\alpha} \) and where the action of a derivation \( \partial \) of \( \mathscr{O}_X \) is
\[ \partial (g f^{-\alpha}) = \left(\partial (g) - g \frac{\alpha \partial(f)}{f}\right) f^{-\alpha}.\]
This \( \mathscr{D}_X \)-module can be endowed with a filtration \( F_k \mathcal{M}(f^{-\alpha}), k \geq 0, \) which makes it a filtered direct summand of a \( \mathscr{D} \)-module underlying a mixed Hodge module on \( X \), see \cite[\S 2]{MP19b}. Therefore, the Hodge ideals of the \( \mathbb{Q} \)-divisor \( \alpha D \) are defined by

\[ F_k \mathcal{M}(f^{-\alpha}) = I_k(\alpha D) \otimes_{\mathscr{O}_X} \mathscr{O}_X(kZ) f^{-\alpha}, \quad \textnormal{for all} \quad k \geq 0. \]

For \( \alpha = 1 \) one recovers the Hodge ideals \( I_k(D) \) of \( D \). Determining the Hodge ideals of a given divisor is a notoriously difficult problem. In the case that the divisor \( D \) has quasi-homogeneous isolated singularities, the Hodge ideals \( I_k(\alpha D) \) have been explicitly described by Zhang \cite{Zha18}. These results are extended to the semi-quasi-homogeneous and Newton non-degenerated cases in \cite{JKYS19}.

\vskip 2mm

For the case of a free divisor \( D \), the Hodge ideals \( I_k(D) \) are determined by Castaño Domínguez, Narváez Macarro and Sevenheck \cite{CNS19} via an algorithm. For the case of the determinant hypersurface, see the work of Perlman and Raicu \cite{PR21}. The first Hodge ideal of the \( \mathbb{Q} \)-divisor \( \alpha D \) coincides with the multiplier ideal \( \mathcal{J}((\alpha - \varepsilon) D), 0 < \varepsilon \ll 1 \), see \cite[Proposition 9.1]{MP19b}. Hence, the Hodge ideals can be seen as a generalization of multiplier ideals. There are algorithms to compute multiplier ideals due to Berkesch and Leykin \cite{BL10} and Shibuta \cite{Shi11} that rely on Gröbner basis techniques in the Weyl algebra.

\jump

In this work, we will use the characterization given in \cite{MP20b} of the Hodge ideals \( I_k(\alpha D) \) associated to a reduced effective divisor \( D \) in terms of the \( V \)-filtration of Kashiwara and Malgrange \cite{Kas83, Mal83}. Given the \( V \)-filtration \( V^\alpha \iota_{+} \mathscr{O}_X, \alpha \in \mathbb{Q} \cap (0, 1], \) on the \( \mathscr{D} \)-module theoretic direct image \( \iota_{+} \mathscr{O}_X \) of \( \mathscr{O}_X \), where \( \iota : X \rightarrow X \times \mathbb{C}, x \mapsto (x, f(x)), \) is the graph embedding of \( f\), \cref{alg1} will compute a set of generators for the \( \mathscr{O}_X \)-module
\[ G_p V^\alpha \iota_{+} \mathscr{O}_X = V^\alpha \iota_{+} \mathscr{O}_X \cap \bigoplus_{j=0}^p \mathscr{O}_X \partial_t^j = \bigg\{ \sum_{j = 0}^p v_j \otimes \partial_t^j \in V^\alpha \iota_{+} \mathscr{O}_X \ \Big|\ v_j \in \mathscr{O}_X \bigg\} \]
for a fixed \( p \in \mathbb{N} \) using Gröbner basis techniques in the Weyl algebra. After \cite[Theorem \(\textnormal{A}'\)]{MP20b}, this determines generators for the Hodge ideals \( I_k(\alpha D), k = 0, \dots, p \). Moreover, by a result of Budur and Saito \cite{BS05}, \cref{alg1} gives also a new procedure to compute the multiplier ideals and the (global) jumping numbers of any effective divisor, not necessarily reduced. Similarly to all general algorithms that depend on Gröbner basis computations the worst case complexity can be doubly exponential in the number of variables.

\jump

This work is organized as follows. In \cref{sec:v-filtration} we review the results related to the theory of \( V \)-filtrations of Kashiwara and Malgrange and the Bernstein-Sato polynomials that will be needed for the main algorithm. In \cref{sec:algorithm}, we present \cref{alg1} and we prove its correctness. Some non-trivial examples of Hodge ideals are included at the end. The algorithms described in this work have been implemented in the computer algebra system \textsc{Singular} \cite{Sing}.

\subsection*{Acknowledgments} The author thanks Michael Perlman for reporting incorrect results of the algorithms as presented in the published version of this paper and for helpful conversations on addressing the issue. For certain examples, the Hodge ideals were shifted due to a mistake in the computation of roots multiplicities of certain \(b\)-functions, see \cref{rmk:errata}.

\section{The \texorpdfstring{\( V \)}{V}-filtration of Kashiwara and Malgrange} \label{sec:v-filtration}

Let \( X \) be a smooth complex variety of dimension \( n \) with structure sheaf \( \mathscr{O}_X \). Let \( \mathscr{D}_X \) denote the sheaf of differential operators on \( X \). The \( V \)-filtration of Kashiwara \cite{Kas83} and Malgrange \cite{Mal83} on a \( \mathscr{D}_X \)-module is defined with respect to a closed subvariety \( Z \subset X \). Through this work we will assume that \( Z \) is a codimension one subvariety globally defined by a regular function \( f \in \Gamma(X, \mathscr{O}_X) \).

\subsection{The smooth case}
When \( Z \) is a smooth subvariety, the \( V \)-filtration on \( \mathscr{D}_X \) along \( Z \) is defined by
\[ V^k \mathscr{D}_X = \{ P \in \mathscr{D}_X \ |\ P \cdot (f)^i \subset (f)^{i + k}\ \textnormal{for all}\ i \in \mathbb{Z}\}   \]
with \( k \in \mathbb{Z} \) and \( (f)^i = \mathscr{O}_X \) for \( i \leq 0 \).

\jump

Since the pair \((X, Z) \) is smooth, one can consider local algebraic coordinates of the form \( x_1, \dots, x_{n-1}, t = f \). Therefore, \( V^k \mathscr{D}_X \) is locally generated over \( \mathscr{O}_X \) by
\[ \prod_{1 \leq i \leq n - 1} \hspace{-9pt} \partial_{x_i}^{\alpha_i} \ t^\nu \partial_t^\mu, \quad \textnormal{with} \quad \nu - \mu \geq k. \]
From this, the \( V \)-filtration on \( \mathscr{D}_X \) along \( Z \) is then an exhaustive decreasing filtration satisfying \( V^i \mathscr{D}_X \cdot V^j \mathscr{D}_X \subseteq V^{i + j} \mathscr{D}_X \) with equality for \( i, j \geq 0 \). In the sequel, \( t \) will always denote a local equation for \( Z \) and \( \partial_t \) a local vector field such that \( [\partial_t, t] = 1 \).

\jump

In general, a \( V \)-filtration on a coherent left \( \mathscr{D}_X \)-module \( \mathcal{M} \) along \( Z \subset X \) is a rational filtration \( (V^\alpha \mathcal{M})_{\alpha \in \mathbb{Q}} \) that is exhaustive, decreasing, discrete and left continuous, such that the following conditions are satisfied:

\begin{enumerate}[i)]
\item Each \( V^\alpha \mathcal{M} \) is a coherent module over \( V^0\mathscr{D}_{X} \).
\item For every \( \alpha \in \mathbb{Q} \), there is an inclusion
\[ t \cdot V^\alpha \mathcal{M} \subseteq V^{\alpha + 1} \mathcal{M} \]
with equality for \( \alpha > 0 \).
\item For every \( \alpha \in \mathbb{Q} \), one has
\[ \partial_t \cdot V^\alpha \mathcal{M} \subseteq V^{\alpha-1} \mathcal{M}. \]
\end{enumerate}
In particular, \( V^i \mathscr{D}_X \cdot V^\alpha \mathcal{M} \subseteq V^{\alpha + i} \mathcal{M} \).
\begin{enumerate}[iv)]
\item For every \( \alpha \in \mathbb{Q} \), the action of \( \partial_t t - \alpha \) is nilpotent on
\[ \textnormal{Gr}_{V}^\alpha := V^\alpha \mathcal{M} / V^{>\alpha} \mathcal{M} \]
where \( V^{> \alpha} \mathcal{M} := \bigcup_{\alpha' > \alpha} V^{\alpha'} \mathcal{M} \).
\end{enumerate}

All conditions are independent of the choice of the local coordinate \( t \). In case such a filtration exists then it is necessarily unique, see \cite[Lemme 3.1.2]{Sai88}. Under reasonable assumptions for the \( \mathscr{D}_X \)-module \( \mathcal{M} \), \( V \)-filtrations do exist.

\begin{theorem}[\cite{Kas83, Mal83}]
Let \( \mathcal{M} \) be a regular holonomic \( \mathscr{D}_X \)-module with quasi-unipotent monodromy around \( Z \). Then, \( \mathcal{M} \) admits a \( V \)-filtration along \( Z \).
\end{theorem}

\jump

\subsection{The graph embedding}
In general, when \( Z \) is singular, one reduces to the smooth case using the graph embedding of \( f \). Namely, let
\[ \iota : X \longrightarrow X \times \mathbb{C}, \quad x \mapsto (x, f(x)) \]
be the closed embedding defined by the graph of \( f \). Define \( Y := X \times \mathbb{C} \) and let \( t \) be the projection on the second factor of \( Y \). This way, \( (t=0) \) is the smooth hypersurface \( X \times \{0\} \) in \( Y \). Given now \( \mathcal{M} \) a \( \mathscr{D}_{X} \)-module, we can consider the \( \mathscr{D} \)-module theoretic direct image by the graph embedding
\[ \iota_{+} \mathcal{M} := \mathcal{M} \otimes_{\mathbb{C}} \mathbb{C}[\partial_t] \]
with the left \( \mathscr{D}_{X \times \mathbb{C}} \)-action given as follows. Let \( m \) be a section of \( \mathcal{M} \),
\begin{enumerate}[i)]
\item \( g \cdot (m \otimes \partial_t^j) = gm \otimes \partial_t^j \),\ for \( g \) a section of \( \mathscr{O}_X \).
\item \( t \cdot (m \otimes \partial_t^j) = fm \otimes \partial_t^j - j m \otimes \partial_t^{j-1} \).
\item \( \partial_t \cdot (m \otimes \partial_t^j) = m \otimes \partial_t^{j+1} \).
\item \( D(m \otimes \partial_t^j) = D(m) \otimes \partial_t^j - D(f)m \otimes \partial_t^{j+1} \),\ for \( D \) a section of \(\textnormal{Der}(\mathscr{O}_X) \),
\end{enumerate}
see for instance \cite[Example 1.3.5]{HTT08}.

\jump

In this situation one can consider the \( V \)-filtration on \( \iota_+ \mathcal{M} \) along \( X \times \{0\} = Z \). Then, define \( V^\bullet \mathcal{M} := \mathcal{M} \cap V^\bullet \iota_+ \mathcal{M} \) for the \( V \)-filtration on \( \mathcal{M} \) along \( Z \). Notice also that in this case \( V^0 \mathscr{D}_Y \) is just \( \mathscr{D}_X \langle t, t \partial_t \rangle \).

\jump

For the case of the \( \mathscr{D}_X \)-module \( \mathscr{O}_{X} \) one has the following alternative description of \( \iota_+ \mathscr{O}_X \). There is an isomorphism of \( \mathscr{D}_{Y} \)-modules
\[ \iota_{+} \mathscr{O}_{X} \cong \mathscr{O}_X[t]_{f-t}/ \mathscr{O}_X[t]. \]
Indeed, denoting by \( \delta \) the class of \( \frac{1}{f-t} \) in \(\mathscr{O}_X[t]_{f-t}/ \mathscr{O}_X[t]\), any section can be uniquely written as
\[ \sum_{j \geq 0} h_j \partial^j_t \delta, \]
with \( h_j \) being sections of \( \mathscr{O}_{X} \) and only finitely many terms being non-zero. Any such section can be identified with \( \sum_{j \geq 0} h_j \otimes \partial_t^j \) and one can check that the \( \mathscr{D}_Y \)-action coincide. Notice that by definition one has \( f \delta = t \delta \).

\jump

Given an arbitrary \( \mathscr{D}_{X} \)-module \( \mathcal{M} \), one recovers the original definition of \( \iota_+ \mathcal{M} \) via the following isomorphism of \( \mathscr{D}_{Y} \)-modules
\[ \iota_{+} \mathcal{M} \cong \mathcal{M} \otimes_{\mathscr{O}_X} \iota_{+} \mathscr{O}_{X} = \bigoplus_{j \geq 0} \mathcal{M} \otimes_{\mathscr{O}_X} \mathscr{O}_{X} \partial_t^j \delta. \]
This description of \( \iota_{+} \mathcal{M} \) leads to the following increasing and exhaustive filtration of \( \mathscr{O}_X \)-modules on \( \iota_{+} \mathcal{M} \) that will be useful in the sequel,
\[ G_k \iota_{+} \mathcal{M} := \bigoplus_{j = 0}^k \mathcal{M} \otimes_{\mathscr{O}_X} \mathscr{O}_X \partial_t^j \delta,  \]
with \( G_k \iota_{+} \mathcal{M} / G_{k - 1} \iota_{+} \mathcal{M} \cong \mathcal{M} \) as \( \mathscr{O}_{X} \)-modules.

\jump

\subsection{Hodge ideals and the \texorpdfstring{\( V \)}{V}-filtration}
The Hodge ideals of a reduced divisor \( D = \textnormal{div}(f) \) can be described in terms of the \( V \)-filtration on \( \iota_+ \mathscr{O}_{X} \). Before presenting the main result from \cite{MP20b} that we will use, it is convenient to define the following polynomials,
\[ Q_i(X) = X(X + 1) \cdots (X + i -1) \in \mathbb{Z}[X].  \]

\begin{theorem}[{\cite[Theorem \(\textnormal{A}'\)]{MP20b}}] \label{thm:mustata-popa}
If \( D \) is a reduced divisor, then for every positive rational number \( \alpha \), and every \( p \geq 0 \), one has
\[ I_p(\alpha D) = \left\{ \sum_{j = 0}^p Q_j(\alpha) f^{p-j} v_j\ \bigg|\ \sum_{j = 0}^p v_j \partial_t^j \delta \in V^\alpha \iota_+ \mathscr{O}_{X} \right\}. \]
\end{theorem}

If one defines \( G_k V^\alpha \iota_+\mathscr{O}_X := G_k \iota_+ \mathscr{O}_X \cap V^\alpha \iota_+ \mathscr{O}_X \) for all \( \alpha \in \mathbb{Q} \), then in order to get generators for the Hodge ideals \( I_k(\alpha D), k = 0, \dots, p, \) it is enough to compute a set of generators for the \( \mathscr{O}_X \)-module \( G_p V^\alpha \iota_+ \mathscr{O}_X \) since \( I_k(\alpha D) \) only depends on the \( \mathscr{O}_X \)-module structure of \( G_p V^\alpha \iota_{+} \mathscr{O}_X \) when \( k = 0, \dots, p \).

\jump

The Hodge ideals \( I_p(\alpha D) \) are a generalization of the multiplier ideals \( \mathcal{J}(\alpha D), \alpha \in \mathbb{Q}_{>0} \) that usually appear in the context of birational geometry, see \cite[\S III]{Laz04-2}.  It is a result of Budur and Saito \cite{BS05} that multiplier ideals have an interpretation in terms of \( \mathscr{D}_X \)-modules via the \( V \)-filtration on \( \mathscr{O}_X \) along \( D \).

\begin{theorem}[{\cite[Thm 0.1]{BS05}}] \label{thm:budur-saito}
Let \( D = \textnormal{div}(f) \) an effective divisor on \( X \), then
\[ \mathcal{J}(\alpha D) = V^{\alpha + \epsilon} \mathscr{O}_X \quad \textnormal{for} \quad 0 < \epsilon \ll 1. \]
\end{theorem}

Indeed, \( I_0(\alpha D) = \mathcal{J}((\alpha - \epsilon) D) \) for \( 0 < \epsilon \ll 1 \), see \cite[Proposition 9.1]{MP19b} and \cref{thm:mustata-popa} can be seen as a generalization of \cref{thm:budur-saito} for the case of reduced divisors.

\subsection{The Bernstein-Sato polynomial}

In order to study the \( V \)-filtration on the \( \mathscr{D}_{Y} \)-module \( \iota_+ \mathscr{O}_{X} \) it is convenient to work on a bigger \( \mathscr{D}_{Y} \)-module where multiplication by \( f \) is bijective, namely \( \iota_+ \mathscr{O}_{X}(* D) \), where \( D = \textnormal{div}(f) \).

\jump

Under the hypothesis that \( f \) acts bijectively on a \( \mathscr{D}_X \)-module \( \mathcal{M} \), that is, \( \mathcal{M} \) has structure of \( \mathscr{O}_X(*D) \)-module, one can show that multiplication by \( t \) is bijective in \( \iota_+ \mathcal{M} \). Hence, for the particular case of \( \mathscr{O}_X \),
\begin{equation} \label{eq:localization}
(\iota_+ \mathscr{O}_X)_t = \iota_+ \mathscr{O}(*D).
\end{equation}

Denote by \( \mathcal{M}_f \) the cyclic \( \mathscr{D}_X[s] \)-submodule of \( \mathscr{O}_X[f^{-1}, s]f^s \) generated by \( f^s \), that is
\[ \mathcal{M}_f := \mathscr{D}_{X}[s]f^s \subseteq \mathscr{O}_{X}[f^{-1}, s]f^s \cong \iota_{+} \mathscr{O}_{X}(* D), \]
where in the last isomorphism the symbolic power \( f^s \) is naturally identified with \( \delta \) and \( s \) acts as \( - \partial_t t \). Notice that, since \( (\partial_t t)^m \delta \) is a section of \( \iota_{+} \mathscr{O}_X \) for all \( m \in \mathbb{N} \), we have the inclusion \( \mathcal{M}_f \subseteq \iota_+ \mathscr{O}_X \).

\jump

The action of \( t \) in \( \mathscr{O}_X[f^{-1}, s]f^s \) is given by the automorphism \( s \mapsto s + 1 \). Since we have the relation \( t s = (s + 1) t \), multiplication by \( t \) leaves invariant \( \mathcal{M}_f \), that is \( t \cdot \mathcal{M}_f \cong \mathscr{D}_{X}[s]f \cdot f^s \subseteq \mathcal{M}_f\). Therefore, \( \mathcal{M}_f \) is in fact a \( \mathscr{D}_X \langle t, s \rangle \)-module. In addition, since \( t^i \partial_t^j \delta = \prod_{k = 1}^j (\partial_t t - (i - k +1)) t^{i - j} \in \mathscr{D}_X[\partial_t t] \delta = \mathcal{M}_f \), one sees that
\begin{equation} \label{eq:localization-t}
(\mathcal{M}_f)_t = (\iota_+ \mathscr{O}_X)_t.
\end{equation}

After Equations \ref{eq:localization} and \ref{eq:localization-t}, it will be convenient to define the following increasing and exhaustive filtration \( T_k \) of \( \mathscr{D}_{X} \langle t, s \rangle \)-modules on \( \iota_{+} \mathscr{O}_X(* D) \) by
\[ T_k(\mathscr{O}_X[f^{-1}, s]f^s) := t^{-k} \mathcal{M}_f = \mathscr{D}_X[s] f^{s-k}. \]

For a general \( \mathscr{D}_X \)-module, one has the following relation that will be applied to \( \mathcal{M} = \mathscr{O}_X(*D) \) later on.

\begin{proposition}[{\cite[Proposition 2.5]{MP20b}}] \label{prop:bijection}
If \( \mathcal{M} \) is a \( \mathscr{D}_{X} \)-module on which \( f \) acts bijectively, then we have an isomorphism of \( \mathscr{D}_{X}\langle t, t^{-1}, s \rangle \)-modules
\[ \Phi : \mathcal{M}[s] f^s \xrightarrow[\quad]{\cong} \iota_{+} \mathcal{M}, \quad u s^j f^s \mapsto u \otimes (-\partial_t t)^j \delta. \]
The inverse isomorphism \( \Psi \) is given by
\[ u \otimes \partial_t^j \delta \mapsto \frac{u}{f^j} Q_j(-s) f^s. \]
\end{proposition}

This setting leads to one of the fundamental objects in the theory of \( \mathscr{D} \)-modules.

\begin{proposition}[\cite{Ber72}] \label{prop:BS}
The action of \( s \) induces an endomorphism
\[ s : \mathcal{M}_f / t \mathcal{M}_f \longrightarrow \mathcal{M}_f / t \mathcal{M}_f, \]
which has a minimal polynomial equal to the Bernstein-Sato polynomial \( b_{f}(s) \) of \( f \).
\end{proposition}

It is a well-known result due to Kashiwara \cite{Kas76} that the roots of the Bernstein-Sato polynomial are negative rational numbers. The existence of the \( V \)-filtration on \( \iota_+ \mathscr{O}_X \) is originally due to Malgrange \cite{Mal83} and Kashiwara \cite{Kas83} using the rationality of the roots of the Bernstein-Sato polynomial. In fact, Malgrange in \cite{Mal83} proves the existence of the \( V \)-filtration on \( \iota_+ \mathscr{O}_X(*D) \). Then, one simply has that \( V^\bullet \iota_+ \mathscr{O}_X = \iota_+ \mathscr{O}_X \cap V^\bullet \iota_+ \mathscr{O}_X(*D)\). Moreover, the following is true.

\begin{lemma}[{\cite[Lemme 3.1.7]{Sai88}}] \label{lemma:saito}
The canonical inclusion \( \iota_+ \mathscr{O}_X \hookrightarrow \iota_+ \mathscr{O}_X(*D) \) induces an equality
\[ V^\alpha \iota_+ \mathscr{O}_X = V^\alpha \iota_+ \mathscr{O}_X(*D) \quad \textnormal{for all} \quad \alpha > 0.\]
\end{lemma}

\section{The algorithm} \label{sec:algorithm}

In this section we shall assume that \( X = \mathbb{A}_\mathbb{C}^n \). Therefore, \( R = \Gamma(X, \mathscr{O}_X) = \mathbb{C}[x_1, \dots, x_n] \) is the polynomial ring and \( D = \Gamma(X, \mathscr{D}_X) = \mathbb{C}[x_1, \dots, x_n] \langle \partial_{x_1}, \dots, \partial_{x_n} \rangle \) is the Weyl algebra. The algorithm presented in this section will make use of the following constructions in computational \( D \)-module theory.

\jump

\subsection{The \texorpdfstring{\( s \)}{s}-parametric annihilator} Let \( f \in R \) be non-constant. The cyclic \( D[s] \)-module \( D[s]f^s \) is isomorphic to \( D[s] / \textnormal{Ann}_{D[s]}f^s \), where \( \textnormal{Ann}_{D[s]}f^s \) is the \( s \)-parametric annihilator of the formal symbol \( f^s \).

\jump

Consider the Malgrange ideal of \( f \),
\[ I_f := D\langle t, \partial_t\rangle \langle f - t, \partial_{x_1} + \frac{\partial f}{\partial x_1} \partial_t, \dots, \partial_{x_n} + \frac{\partial f}{\partial x_n} \partial_t \rangle. \]
Then, the \( s \)-parametric annihilator of \( f \) equals \( \left.I_f \cap D[\partial_t t]\right|_{\partial_t t = -s} \), see for instance \cite[Theorem 5.3.4]{SST00}. Moreover, such elimination of variables can be computed using Gröbner basis techniques in the Weyl algebra, see \cite[Algorithm 5.3.6]{SST00}. There are similar ways to compute generators for \( \textnormal{Ann}_{D[s]}f^s \) due to Briançon and Maisonobe \cite{BM02} that usually perform better due to the need to eliminate fewer variables.

\jump

By \cref{prop:BS}, the Bernstein-Sato polynomial \( b_f(s) \) of \( f \) can then be computed as the minimal polynomial of \( s \) acting on
\[ \frac{D[s] f^s}{D[s] \langle f \rangle f^{s}}\cong \frac{D[s]}{\textnormal{Ann}_{D[s]} f^s + D[s] \langle f \rangle}. \]
That is, \( \langle b_{f}(s) \rangle = (\textnormal{Ann}_{D[s]} f^s + D[s] \langle f \rangle) \cap \mathbb{C}[s] \). This intersection can either be computed by standard Gröbner basis elimination techniques or taking advantage of the fact that \( \mathbb{C}[s] \) is a principal subalgebra of \( D[s] \), see \cite{ALMM09}.

\subsection{Modulo operation}

The following construction from \cite{Lev05} provides an efficient way of computing the kernel of morphisms of \( D \)-modules. Let \( N, M \) be left submodules of the free submodules \( D^n = \sum_{i = 0}^n D e_i \) and \( D^m \), respectively. Consider,
\[ \phi : D^n/N \longrightarrow D^m/M, \qquad e_i \mapsto \Phi_i, \]
a morphism of left \( D \)-modules given by the matrix \( \Phi = (\Phi_1 \ |\ \cdots \ |\ \Phi_n) \in D^{m \times n} \). Define the matrix
\def\arraystretch{1.1}
\[ Y = \left(\begin{array}{@{}c|c@{}}
    \Phi & M \\\hline
    \Id_{n \times n} & \textbf{0}
  \end{array}\right) \in D^{(m+n)\times(n+k)},
\]
where \( k \in \mathbb{N} \) is the number of generators of \( M \). Let \( Z = D^{n+m} Y \cap \bigoplus_{i = m + 1}^{n + m} D e_i \), this intersection can be computed with standard elimination of components. Then, one has  \[ \Ker \phi \cong (Z + N) / N,\] see \cite[Lemma 9]{Lev05}. This method avoids the computation of unnecessary syzygies and computes only those relevant to \( \Phi  \). Keeping the same notation from \cite{Lev05} and \textsc{Singular} \cite{Sing}, we will denote the operation that computes a system of generators for \( Z \) from generators of \( \Phi \) and \( M \) as \( \textsc{Modulo}(\Phi, M) \).

\subsection{The algorithm}

We present next the main contribution of this work. The following algorithm computes a set of generators of the \( \mathscr{O}_X \)-modules \( G_p V^\alpha \iota_+ R \) for any \( \alpha \in \mathbb{Q} \cap (0, 1] \). This is of course feasible because, by the definition of \( V \)-filtration, there is only a finite number of different such \( \mathscr{O}_X \)-modules when \( \alpha \) ranges in \( \mathbb{Q} \cap (0, 1] \). The algorithm is inspired by the construction of the \( V \)-filtration on \( \iota_{+}\mathscr{O}_X(*D) \) by Malgrange \cite{Mal83}.

\jump

The main ideas behind the algorithm are the following. For \( \alpha \in \mathbb{Q} \cap (0, 1] \), we will actually compute \( G_p V^\alpha \iota_+ R[f^{-1}] \). By \eqref{eq:localization-t}, any element of \( \iota_+ R[f^{-1}] \cong R[s, f^{-1}]f^s \) lies in \( T_k \iota_{+} R[f^{-1}] = t^{-k} \mathcal{M}_f = D[s]f^{s-k} \) for some \( k \in \mathbb{N} \). Then, consider the endomorphism
\[ s : \frac{D[s]f^{s-p}}{t D[s]f^{s}} \longrightarrow \frac{D[s]f^{s-p}}{t D[s]f^{s}} \]
with minimal polynomial \( b^{(p)}_f(s) \). Since the action of \( t \) is bijective, the roots of \( b^{(p)}_f(s) \) are of the form \( \alpha + k\) for \( \alpha \) a root of \( b_f(s) \) and \( k = 0, \dots, p \). Hence, one has the decomposition
\begin{equation} \label{eq:sec3-eq1}
  \frac{D[s]f^{s - p}}{t D[s] f^s} = \bigoplus_{\lambda} P_\lambda,
\end{equation}
where the sum ranges over all the roots \( \lambda \) of \( b^{(p)}_f(-s) \) and \( P_\lambda \) is the submodule on which \( s + \lambda \) acts nilpotently. Consider, for \( \alpha \in \mathbb{Q} \cap (0, 1] \), the \( D[s] \)-submodule \( W_\alpha \) satisfying \( t D[s] f^s \subseteq W_\alpha \subseteq D[s] f^{s-p} \) and
\begin{equation} \label{eq:sec3-eq2}
  \frac{W_\alpha}{t D[s] f^s} = \bigoplus_{\lambda > \alpha} P_\lambda.
\end{equation}
Then, \( G_p V^\alpha \iota_+ R[f^{-1}] \) can be computed from \( W_\alpha \cap R[s] \) by taking all the elements of degree less than or equal to \( p \) in \( s \). The full details of the correctness of the algorithm are delayed until \cref{thm:alg1} below.

\jump

All the sets appearing in \cref{alg1} below are assumed to be ordered.

\newpage

\begin{algorithm}(Generators \( G_p V^\alpha \iota_+ R \)) \label{alg1}
\begin{algorithmic}[1]
\REQUIRE A reduced \(f \in R \) and \( p \in \mathbb{N} \).
\ENSURE A basis of the \( R \)-module \( G_p V^\alpha \iota_{+}R  \) for \( \alpha \in \mathbb{Q} \cap (0, 1] \).
\jump
\jump
\STATE \( G \gets \) Gröbner basis of \( \textnormal{Ann}_{D[s]}f^s \) w.r.t. any monomial ordering
\STATE \( J^{(p)} \gets \) Gröbner basis of \( D[s] \langle \left.G\right|_{s \mapsto s - p}, f^{p + 1} \rangle \) w.r.t. an elimination order for \( \underline{x}, \underline{\partial} \)
\STATE \( b^{(p)}_f(s) \gets \) generator of \( J^{(p)} \cap \mathbb{C}[s] \)
\STATE \( \rho^{(p)}_f \gets \{ (\alpha_i, n_i) \in \mathbb{Q} \times \mathbb{N} \ |\ b^{(p)}_f(s) = (s - \alpha_1)^{n_1} \cdots (s - \alpha_r)^{n_r},\, \alpha_1 < \cdots < \alpha_r < 0\} \)
\jump
\jump
\FOR{\( (\alpha, n_\alpha) \in \{ (\alpha_i, n_i) \ |\ (\alpha_i, n_i) \in \rho^{(p)}_f, \alpha_ i < 0 \}\)}
  \STATE \( K_{\alpha} \gets \textsc{Modulo}(s - \alpha, J^{(p)}) \)
  \FOR{\( i = 1, \dots, n_\alpha - 1 \)}
  \STATE \( K_{\alpha} \gets \textsc{Modulo}(s - \alpha, K_\alpha) \)
  \ENDFOR
\ENDFOR
\jump
\jump
\STATE \( \alpha' \gets -\infty \)
\STATE \( G_{\alpha'} \gets \bigcup_{\lambda < -1} K_\lambda \)
\FOR{\( \alpha \in \{ \alpha_i + k \ |\ (\alpha_i, n_i) \in \rho^{(p)}_f, \alpha_i \in [-1, 0) \} \)}
  \STATE \( G_\alpha \gets \) Gröbner basis of \( D[s] \langle G_{\alpha'}, K_\alpha \rangle \) w.r.t. an elimination order for \( \underline{\partial} \)
  \STATE \( H_\alpha \gets \) Gröbner basis of \( G_\alpha \cap R[s] \) w.r.t. an elimination order for \( s \)
  \STATE \( H_\alpha^{(p)} \gets \big\{ \sum_{j = 0}^p h_j (-\partial_t t)^j \ |\ \sum_{j = 0}^p s^j h_j \in H_\alpha \big\} \)
  \STATE \( B_\alpha^{(p)} \gets \big\{ \sum_{j = 0}^p h_j' \partial_t^j f^j \cdot f^{-p} \ |\ \sum_{j = 0}^p h_j' \partial_t^j t^j \in H_{\alpha}^{(p)}\big\}  \)
  \STATE \(\alpha' \gets \alpha\)
\ENDFOR
\jump
\jump
\RETURN \( B_\alpha^{(p)} \)
\end{algorithmic}
\end{algorithm}

\begin{theorem} \label{thm:alg1}
\cref{alg1} is correct.
\end{theorem}
\begin{proof}
In Lines 1--4, the algorithm starts by computing a Gröbner basis \( G \) of the \( s \)-parametric annihilator \( \textnormal{Ann}_{D[s]} f^s \) and the Bernstein-Sato polynomial \( b^{(p)}_{f}(s) \) of \( f \), using the standard methods discussed at the beginning of this section.

\jump

In Line 2, \( J^{(p)} \) is a basis of the submodule of \( D[s] \) giving a presentation of \( D[s]f^{s -p}/tD[s]f^s \). Indeed,
\begin{equation} \label{eq:thm-eq1}
\frac{D[s]f^{s-p}}{tD[s]f^s} = \frac{t^{-p} D[s] f^{s}}{tD[s]f^s} = \frac{t^{-p} D[s] f^{s}}{t^{p+1} \cdot t^{-p}D[s]f^s} \cong \frac{D[s]}{\textnormal{Ann}_{D[s]} f^{s-p} + D[s]\langle f^{p+1} \rangle}.
\end{equation}
Notice that, since we already have a basis \( G \) of \( \textnormal{Ann}_{D[s]} f^s \), after the substitution \( s \mapsto s - p \) in \( G \) one gets a basis of \( \textnormal{Ann}_{D[s]} f^{s-p} \). 

\jump

The loop that starts in Line 5 iterates over the strictly negative roots of \( b_f^{(p)}(s)\). For each root \( \alpha \) of multiplicity \( n_\alpha \), the next steps, Lines 6--9, compute a basis \( K_\alpha \) of the kernel of the morphism

\[ (s - \alpha)^{n_\alpha} : \frac{D[s]}{D[s]\langle J^{(p)}\rangle} \longrightarrow \frac{D[s]}{D[s]\langle J^{(p)} \rangle}.  \]
Then, the submodule \( P_{-\alpha} \) from \eqref{eq:sec3-eq1} associated to \( s - \alpha \) is isomorphic to \( D[s] \langle K_\alpha \rangle / D[s] \langle J^{(p)} \rangle \). The computation is done inductively by computing the kernel of \( (s - \alpha)^{i+1} \) from the kernel of \( (s - \alpha)^i \). This strategy makes the computation much more efficient in practice.

\jump

The last part of the Algorithm works as follows. The sets \( G_\alpha = \bigcup_{\lambda < \alpha} K_\lambda \) form a basis of the \( D[s] \)-submodules \( W_{-\alpha} \) from \eqref{eq:sec3-eq2}. Indeed, by construction every submodule \( D[s] \langle K_\alpha \rangle \) contains \( D[s] \langle J^{(p)} \rangle \), and hence it contains \( D[s] \langle f^{p+1} \rangle \). Then, since \( t D[s] f^s = t^{p + 1} D[s]f^{s-p} \), the submodules \( tD[s]f^s \subseteq W_\alpha \subseteq D[s]f^{s-p} \) are isomorphic to
\[ \frac{D[s] \langle f^{p+1} \rangle}{\textnormal{Ann}_{D[s]} f^{s-p}} \subseteq \frac{D[s] \langle G_\alpha \rangle}{\textnormal{Ann}_{D[s]} f^{s-p}} \subseteq \frac{D[s]}{\textnormal{Ann}_{D[s]} f^{s-p}}.\]

The set \( G_\alpha \cap R[s] \) in Line 15 forms a Gröbner basis of the \( R[s] \)-module \( W_{-\alpha} \cap R[s] \) since one has \( \textnormal{Ann}_{D[s]} f^{s} \cap R[s] = 0 \). It remains to show how an \( R \)-basis of \( G_p V^{-\alpha} R[f^{-1}] \) is obtained from \( W_{\alpha} \cap R[s] \). Given an elimination order \( \succcurlyeq_s \) for \( s \) in \( R[s] \), we have that \( R[s]_{\succcurlyeq_s} = R_{\succcurlyeq_s'}[s] \) where \( \succcurlyeq_s' \) is the monomial order induced by \( \succcurlyeq_s \) in \( R \). Therefore, for any \( f \in R[s] \), if the leading monomial of \( f \) with respect to \( \succcurlyeq_s \) is in \( R \cdot s^i  \), then \( f \in \bigoplus_{j=0}^i R \cdot s^j \). Consequently, a Gröbner basis \( H_\alpha \) of \( G_\alpha \cap R[s] \) with respect to \( \succcurlyeq_s \) gives a basis of \( G_\alpha \cap R[s] \) as \( R \)-submodule of \( \bigoplus_{j = 0}^i R \cdot s^j \).

\jump

Taking the elements of degree at most \( p \) in \( s \) from \( H_\alpha \) and making the substitution \( s = - \partial_t t \) leaves us with elements of the form \( \sum_{j = 0}^p h_j' \partial_t^j t^j \). Line 17 denotes by \( H_\alpha^{(p)} \) the set of such elements. The isomorphism
\[ \frac{D[s]}{\textnormal{Ann}_{D[s]}f^{s-p}} \cong D[s] f^{s-p} \cong t^{-p} D[s] f^s \cong t^{-p} D[\partial_t t] \delta  \]
is given by sending the class of \( 1 \) to \( t^{-p} \delta \). Therefore, since \( t \delta = f \delta \), the claimed \( R \)-basis of \( G_{p} V^{-\alpha} \iota_{+} R \) will be given by the set \( B_\alpha^{(p)} \) obtained from \( H_{\alpha}^{(p)} \) after substituting \( t = f \) and multiplying by \( f^{-p} \). Notice that the elements of \( B_\alpha^{(p)} \) are well-defined in \( \iota_{+} R[f^{-1}] \). However, for \( \alpha < 0 \), \(B_\alpha^{(p)}\) is actually in \( \iota_{+} R \), that is, the division by \( f^p \) in Line 18 give rise to no rational functions. Indeed, this will follow from \cref{lemma:saito} once we show that \( B_\alpha^{(p)} \subseteq G_p V^{-\alpha} \iota
_{+} R[f^{-1}] \).

\jump

Let \( \alpha' \) be a rational number from the set in Line 14 such that \( \alpha' < \alpha \). Then, we have \( B_{\alpha'}^{(p)} \subseteq B_{\alpha}^{(p)} \). If \( \alpha' \) is the largest of such numbers, then \( \partial_t t + \alpha \) acts nilpotently on \( R\langle B_\alpha^{(p)} \rangle/R\langle B_{\alpha'}^{(p)}\rangle \) by construction. It remains to show that \( t \cdot R\langle B_{\alpha}^{(p)} \rangle \subseteq R \langle B_{\alpha - 1}^{(p)} \rangle \) and \( \partial_t \cdot R \langle B_{\alpha}^{(p)} \rangle \subseteq R \langle B_{\alpha+1}^{(p+1)} \rangle \). Let \( u \in P_{-\alpha} \cong D[s] \langle K_\alpha \rangle / \textnormal{Ann}_{D[s]} f^{s-p} \), then the first inclusion follows from the equality
\[ t \cdot (s - \alpha)^n u = (s - \alpha + 1)^n t \cdot u = 0. \]
For the second one, notice that \( s W_\alpha \subseteq W_\alpha \). Then, the second inclusion follows by the identity \( \partial_t = t^{-1}(\partial_t t - 1) \). This proves that \( B_\alpha^{(p)} \subseteq G_p V^{-\alpha} \iota_{+} R \).

\jump

In order to conclude the proof, the last thing remaining is to show the reverse inclusion \( G_{p} V^{-\alpha} \iota_{+} R \subseteq R \langle B_{\alpha}^{(p)} \rangle \). But this follows from the fact that the isomorphism \( \Psi \) from \cref{prop:bijection} sends the elements of \( G_p \iota_+ R \) to \( T_p \iota_{+} R[f^{-1}] = D[s] f^{s-p} \).
\end{proof}

A straightforward application of \cref{thm:mustata-popa} gives the following corollary.

\begin{corollary} \label{cor:corollary1}
For any \( p \in \mathbb{N} \) and \( f \in R \) reduced, \cref{alg1} computes generators for the Hodge ideals \( I_{k}(f^\alpha), \alpha \in \mathbb{Q} \cap (0, 1], k = 0, \dots, p \).
\end{corollary}

In contrast with \( G_p V^{\alpha} \iota_{+} R[f^{-1}] \), assuming that \( p \in \mathbb{N} \) is fixed, the set of ideals \( I_p(f^\alpha) \) for \( \alpha \in \mathbb{Q} \cap (0, 1] \) is not finite since the Hodge ideals depend on \( \alpha \) even if the \( V \)-filtration does not, i.e. \( V^\alpha \iota_+ R = V^{\alpha + \epsilon} \iota_{+} R \) for \( 0 < \epsilon \ll 1 \). To remedy this and still have a finite output, one can work on a transcendental extension \( R(\alpha) = \mathbb{C}(\alpha)[x_1, \dots, x_n] \) of the base field and compute with the polynomials \( Q_i(\alpha) \) from \cref{thm:mustata-popa} symbolically.

\jump

\begin{remark}
In practice, for a fixed \( \alpha \in \mathbb{Q} \cap (0, 1] \), it is enough to compute the Hodge ideals \( I_k(f^{\alpha}) \) for \( k = 0, \dots, l, \) where \( l \) is the generating level of the Hodge filtration on \( \mathcal{M}(f^{-\alpha}) \). Recall that the generating level of any \( \mathscr{D}_X \)-module \( (\mathcal{M}, F_\bullet) \) equipped with a good filtration is the smallest integer \( l \) such that
\[ F_k \mathscr{D}_X \cdot F_l \mathcal{M} = F_{l + k} \mathcal{M} \quad \textnormal{for all} \quad k \geq 0. \]
Therefore, in our case, for \( k > l \), \( I_k(f^\alpha) \) is generated by \( f \cdot I_{k-1}(f^\alpha) \) and
\begin{equation} \label{eq:equation3}
 \{ f D(h) - (\alpha + k)h D(f) \ |\ h \in I_{k-1}(f^\alpha),\, D \in \textnormal{Der}_{\mathbb{C}}(R)\}.
\end{equation}
see \cite[\S 10]{MP19b}. In addition, for a reduced \( f \in \mathscr{O}_{X}(X) \), the Hodge filtration on \( \mathcal{M}(f^{-\alpha}) \) is known to be generated at level \( n - \lceil \tilde{\alpha}_f + \alpha \rceil \), see \cite[Theorem E]{MP20a}, where \( \tilde{\alpha}_f \) is the minimal exponent of \( f \), that is, the smallest root of \( b_f(-s)/(1-s) \). Conversely, notice that from \cref{cor:corollary1} and \eqref{eq:equation3} one can compute the minimal generating level of the Hodge filtration on \( \mathcal{M}(f^{-\alpha}) \).
\end{remark}

After \cref{thm:budur-saito}, we also obtain a new algorithm to compute the multiplier ideals and the (global) jumping numbers of any effective divisor, not necessarily reduced.

\begin{corollary}
For any \( f \in R \), \cref{alg1} computes a set of generators of the multiplier ideals \( \mathcal{J}(f^\alpha), \alpha \in \mathbb{Q} \cap (0, 1), \) and the jumping numbers of \( f \).
\end{corollary}

\begin{remark} \label{rmk:errata}
In the published version of this paper, the output of \cref{alg1} could be incorrect since it was erroneously assumed that the roots multiplicities of \(b_{f}^{(p)}(s)\) could be computed from those of \(b_f(s)\). As a consequence, for certain examples, the Hodge ideals \(I_k(f^\alpha)\) could appear shifted with respect to the parameter \(k\).
\end{remark}

\subsection{Examples} Let us show some non-trivial examples of Hodge ideals computed with \cref{alg1}. Even though the algorithm can compute more complex examples, for the sake of space the examples below have been chosen with a small degree and number of variables.

\begin{example}
Let \( f = x^5 + y^5 + x^2y^2 \). This is perhaps the simplest plane curve which is not quasi-homogeneous or a \( \mu \)-constant deformation of a quasi-homogeneous singularity.

\def\arraystretch{1.3}
\begin{table}[!ht]
\centering
\begin{tabular}{|c|c|c|}
\hline
\(\alpha\) & \(I_0(f^\alpha)\) & \(I_1(f^\alpha)\) \\[1.5pt] \hline
\hline
\(\frac{1}{10}\)  &  \(R\) & \((x^3,x^2y,xy^2,y^3)\) \\[1.5pt] \hline
\(\frac{3}{10}\)  &  \(R\) & \((x^4,x^2y,xy^2,y^4)\) \\[1.5pt] \hline
\(\frac{1}{2}\)\tiny{(2)} &  \(R\) & \((5x^4+2xy^2,x^3y,x^2y^2,xy^3,5y^4+2x^2y)\) \\[1.5pt] \hline
\(\frac{7}{10}\)  &  \((x,y)\) & \(((5\alpha-2)x^5+(2\alpha-1)x^2y^2,x^3y,xy^3,(5\alpha-2)y^5+(2\alpha-1)x^2y^2)\) \\[1.5pt] \hline
\(\frac{9}{10}\)  &  \((x^2,xy,y^2)\) & \((x^6,x^4y,x^3y^2,x^2y^3,xy^4,y^6)\) \\[1.5pt] \hline
\(1\)\tiny{(2)} &  \((x^3,xy,y^3)\) & \((x^6-(2\alpha-1)x^3y^2,x^5y,x^4y^2,x^3y^3,x^2y^4,xy^5,y^6-(2\alpha-1)x^2y^3)\) \\[1.5pt] \hline
\end{tabular}
\jump
\caption{Hodge ideals \( I_p(f^\alpha), p = 0, 1, \alpha \in \mathbb{Q} \cap (0, 1]  \) for \( f = x^5 + y^5 + x^2y^2 \).}\label{table1}
\end{table}
\end{example}

The subscripts in the first columns of Tables \ref{table1}, \ref{table2} and \ref{table3} denote the nilpotency index of \( \partial_t t - \alpha \) on \( \textnormal{Gr}_V^\alpha \).

\begin{example}
Let \( f_\lambda = (y^2 - x^3)(y^2 + \lambda x^3),\, \lambda \in \mathbb{C}^* \). The parameter \( \lambda \in \mathbb{C}^* \) is an analytical invariant of the singularity defined by \( f_\lambda \) at the origin. That is, for different values of \( \lambda \in \mathbb{C}^* \), the singularities \( f_\lambda \) are not analytically equivalent. Set
\jump
\def\arraystretch{1.3}
\begin{table}[!ht]
\begin{tabular}{|c|c|c|}
\hline
\(\alpha\) & \(I_0(f_\lambda^\alpha)\) & \(I_1(f_\lambda^\alpha)\) \\[1.5pt] \hline
\hline
\(\frac{1}{12}\)  &  \(R\) & \((y^3,xy^2,x^3y,x^4)\) \\[1.5pt] \hline
\(\frac{1}{6}\)  &  \(R\) & \((y^3,x^2y^2,x^3y,x^5)\) \\[1.5pt] \hline
\(\frac{1}{4}\)  &  \(R\) & \((y^4,xy^3,x^2y^2,(\lambda-1)x^3y+2y^3,x^5)\) \\[1.5pt] \hline
\(\frac{1}{3}\)  &  \(R\) & \((y^4,xy^3,(\lambda-1)x^3y+2y^3,2\lambda x^5 - (\lambda-1)x^2y^2)\) \\[1.5pt] \hline
\(\frac{5}{12}\)  &  \(R\) & \((y^4,(\lambda-1)x^3y+2y^3,x^2y^3,2\lambda x^5 - (\lambda-1)x^2y^2)\) \\[1.5pt] \hline
\(\frac{7}{12}\)  &  \((x, y)\) & \((y^5,xy^4,x^2y^3,(\lambda-1)x^3y^2+2y^4,(\lambda-1)x^4y+2xy^3,\lambda x^6+y^4)\) \\[1.5pt] \hline
\(\frac{2}{3}\)  &  \((y, x^2)\) & \((y^5,x^2y^3,(\lambda-1)x^4y^2+2xy^4,x^5y,h_1)\) \\[1.5pt] \hline
\(\frac{3}{4}\)  &  \((x^2, xy, y^2)\) & \((y^5,x^2y^4,x^3y^3,(\lambda-1)x^4y^2+2xy^4,(\lambda-1)x^5y+2x^2y^3,\lambda x^7+xy^4)\) \\[1.5pt] \hline
\(\frac{5}{6}\)  &  \((y^2, xy, x^3)\) & \((y^5,x^2y^4,x^3y^3,x^5y^2,x^6y,xh_1)\) \\[1.5pt] \hline
\(\frac{11}{12}\)  &  \((y^2, x^2y, x^3)\) & \((y^6,xy^5,x^2y^4,h_2,x^5y^2,h_3,x^8)\) \\[1.5pt] \hline
\(1\)\tiny{(2)}  &  \((y^3, xy^2, x^2y, x^4)\) & \((y^6,xy^5,x^3y^4,x^4y^3,x^6y^2,h_4)\) \\[1.5pt] \hline
\end{tabular}
\jump
\caption{Hodge ideals \( I_p(f_\lambda^\alpha), p = 0, 1, \alpha \in \mathbb{Q} \cap (0, 1] \) for \( f = (y^2-x^3)(y^2-\lambda x^3) \).}\label{table2}
\end{table}
\[ h_1 := \lambda x^6+((2\alpha -1)\lambda -2\alpha +1)x^3y^2+(4\alpha -1)y^4 \]
\[ h_2 := ((\alpha-1)\lambda^2-2\lambda\alpha+\alpha-1)x^3y^3+((2\alpha-1)\lambda-2\alpha+1)y^5 \]
\[ h_3 := ((\alpha-1)\lambda^2-2\lambda\alpha+\alpha-1)x^6y - (4\alpha-2)y^5 \]
\[ h_4 := x^7y+\lambda x^8 + ((2\alpha-1)\lambda-2\alpha+1)x^5y^2+(4\alpha-1)x^2y^4 \]
\end{example}

\FloatBarrier

\begin{example}
Let \( f = x^3 + y^3 + z^3 + xyz \). Then, the pair \( (\mathbb{A}^3_\mathbb{C}, \textnormal{div}(f)) \) is log-canonical. Therefore, the multiplier ideals are trivial. Hence, the Hodge ideals provide a first non-trivial invariant of the singularity. Since this example is an isolated quasi-homogeneous singularity the Hodge ideals were already determined in \cite{Zha18}.
\def\arraystretch{1.3}
\begin{table}[!ht]
\begin{tabular}{|c|c|c|c|}
\hline
\(\alpha\) & \(I_0(f^\alpha)\) & \(I_1(f^\alpha)\) & \(I_2(f^\alpha)\) \\[1.5pt] \hline
\hline
\(\frac{1}{3}\) & \(R\) & \( (x,y,z) \) & \( (x^2y,x^2z,xy^2,zy^2,yz^2,xz^2,y^3-z^3,x^3-z^3,xyz+3z^3,z^4) \) \\[1.5pt] \hline
\(\frac{2}{3}\) & \(R\) & \( (x,y,z)^2 \) & \( (x, y, z)^4 \) \\[1.5pt] \hline
\(1\)\tiny{(2)}  &  \(R\) & \((z^3,yz^2,xz^2,\) & \((z^5,yz^4,xz^4,y^2z^3,xyz^3,x^2z^3,y^3z^2,xy^2z^2,x^2yz^2,\) \\[1.5pt]
 & & \( xz+3y^2, \) & \( (3\alpha+56)x^3z+y^3z-53\alpha xyz^2-(81\alpha+28)z^4, \) \\[1.5pt]
 & & \( xy+3z^2, \) & \( 27\alpha y^4+18\alpha xy^2z+(3\alpha+56)x^2z^2, \) \\[1.5pt]
 & & \( yz+3x^2) \) & \( 27\alpha xy^3+9\alpha x^2yz+(81\alpha+28)y^2z^2+xz^3, \) \\[1.5pt]
 & & & \( (3\alpha+56)x^2y^2+18\alpha xyz^2+27\alpha z^4, \) \\[1.5pt]
 & & & \( 27\alpha x^3y+9\alpha xy^2z+(81\alpha+28)x^2z^2+yz^3, \) \\[1.5pt]
 & & & \( 27\alpha x^4+18\alpha x^2yz+(3\alpha+56)y^2z^2) \) \\[1.5pt] \hline
\end{tabular}
\jump
\caption{Hodge ideals \( I_p(f^\alpha), p = 0, 1, 2, \alpha \in \mathbb{Q} \cap (0, 1] \) for \( f = x^3 + y^3 + z^3 + xyz \).}\label{table3}
\end{table}
\end{example}

\bibliography{serials_short,references}

\providecommand{\bysame}{\leavevmode\hbox to3em{\hrulefill}\thinspace}
\providecommand{\MR}{\relax\ifhmode\unskip\space\fi MR }
\providecommand{\MRhref}[2]{%
  \href{http://www.ams.org/mathscinet-getitem?mr=#1}{#2}
}
\providecommand{\href}[2]{#2}
\begin{thebibliography}{ALMM09}

\bibitem[ALMM09]{ALMM09}
D.~Andres, V.~Levandovskyy, and J.~Mart\'in-Morales, \emph{Principal intersection and {B}ernstein-{S}ato polynomial of an affine variety}, Proceedings of the 2009 International Symposium on Symbolic and Algebraic Computation, ACM, New York, NY, 2009, pp.~231--238.

\bibitem[Ber72]{Ber72}
J.~Bernstein, \emph{The analytic continuation of generalized functions with respect to a parameter}, Funct. Anal. Appl. \textbf{6} (1972), no.~4, 26--40.

\bibitem[BL10]{BL10}
Ch. Berkesch and A.~Leykin, \emph{Algorithms for {B}ernstein-{S}ato polynomials and multiplier ideals}, Proceedings of the 2010 International Symposium on Symbolic and Algebraic Computation, ACM, New York, NY, 2010, pp.~99--106.

\bibitem[BM02]{BM02}
J.~Briançon and Ph. Maisonobe, \emph{Remarques sur l’idéal de {B}ernstein associé à des polynômes}, Preprint Université de Nice Sophia-Antipolis (2002), no.~650.

\bibitem[BS05]{BS05}
N.~Budur and M.~Saito, \emph{Multiplier ideals, {$V$}-filtration, and spectrum}, J. Algebraic Geom. \textbf{14} (2005), no.~2, 269--282.

\bibitem[DGPS21]{Sing}
W.~Decker, G.-M. Greuel, G.~Pfister, and H.~Schönemann, \emph{\textsc{Singular} \textup{\texttt{4.2.0}} --- {A} computer algebra system for polynomial computations}, Available at \url{http://www.singular.uni-kl.de}, 2021.

\bibitem[DMS19]{CNS19}
A.~Casta{\~n}o Domínguez, L.~Narváez Macarro, and Ch. Sevenheck, \emph{{H}odge ideals of free divisors}, preprint arXiv:1912.09786.

\bibitem[HTT08]{HTT08}
R.~Hotta, K.~Takeuchi, and T.~Tanisaki, \emph{{$D$}-modules, {P}erverse {S}heaves and {R}epresentation {T}heory}, Progr. Math., no. 263, Birkhäuser, Basel, 2008.

\bibitem[JKYS19]{JKYS19}
S.-J. Jung, I.-K. Kim, Y.~Yoon, and M.~Saito, \emph{Hodge ideals and spectrum of isolated hypersurface singularities}, preprint arXiv:1904.02453.

\bibitem[Kas76]{Kas76}
M.~Kashiwara, \emph{${B}$-functions and holonomic systems}, Invent. Math. \textbf{38} (1976), no.~1, 33--53.

\bibitem[Kas83]{Kas83}
\bysame, \emph{Vanishing cycle sheaves and holonomic systems of differential equations}, Algebraic geometry (Tokyo/Kyoto, 1982), Lecture Notes in Math., no. 1016, Springer, Berlin, 1983, pp.~134--142.

\bibitem[Laz04]{Laz04-2}
R.~Lazarsfeld, \emph{Positivity in algebraic geometry. {II}. {P}ositivity for vector bundles, and multiplier ideals}, Ergeb. Math. Grenzgeb. (3), no.~49, Springer, Berlin, 2004.

\bibitem[Lev05]{Lev05}
V.~Levandovskyy, \emph{On preimages of ideals in certain non-commutative algebras}, Computational commutative and non-commutative algebraic geometry (Chisinau, 2005) (S.~Cojocaru, G.~Pfister, and V.~Ufnarovski, eds.), NATO Sci. Ser. III Comput. Syst. Sci., no. 196, IOS, Amsterdam, 2005, pp.~44--62.

\bibitem[Mal83]{Mal83}
B.~Malgrange, \emph{Polynôme de {B}ernstein-{S}ato et cohomologie évanescente}, Analysis and topology on singular spaces, {II}, {III} (Paris), Astérisque, vol. 101-102, Soc. Math. France, 1983, pp.~243--267.

\bibitem[MP19a]{MP19}
M.~Musta\c{t}\v{a} and M.~Popa, \emph{Hodge ideals}, Mem. Amer. Math. Soc. \textbf{262} (2019), no.~1268.

\bibitem[MP19b]{MP19b}
\bysame, \emph{Hodge ideals for {$\mathbb{Q}$}-divisors: birational approach}, J. Éc. polytech. Math. \textbf{6} (2019), 283--328.

\bibitem[MP20a]{MP20a}
\bysame, \emph{Hodge filtration, minimal exponent, and local vanishing}, Invent. Math. \textbf{220} (2020), 453--478.

\bibitem[MP20b]{MP20b}
\bysame, \emph{Hodge ideals for {$\mathbb{Q}$}-divisors, {$V$}-filtration, and minimal exponent}, Forum Math. Sigma \textbf{8} (2020), no.~e19, 41 pp.

\bibitem[PR21]{PR21}
M.~Perlman and C.~Raicu, \emph{{H}odge ideals for the determinant hypersurface}, Selecta Math. (N.S.) \textbf{21} (2021), no.~1.

\bibitem[Sai88]{Sai88}
M.~Saito, \emph{Modules de {H}odge polarisables}, Publ. Res. Inst. Math. Sci. \textbf{24} (1988), no.~6, 849--995.

\bibitem[Sai90]{Sai90}
\bysame, \emph{Mixed {H}odge modules}, Publ. Res. Inst. Math. Sci. \textbf{26} (1990), no.~2, 221--333.

\bibitem[Sai93]{Sai93}
\bysame, \emph{On $b$-function, spectrum and rational singularity}, Math. Ann. \textbf{295} (1993), no.~1, 51--74.

\bibitem[Shi11]{Shi11}
T.~Shibuta, \emph{Algorithms for computing multiplier ideals}, J. Pure Appl. Algebra \textbf{215} (2011), no.~12, 2829--2842.

\bibitem[SST00]{SST00}
M.~Saito, B.~Sturmfels, and N.~Takayama, \emph{Gröbner deformations of hypergeometric differential equations}, Algorithms Comput. Math., no.~6, Springer, Berlin, 2000.

\bibitem[Zha18]{Zha18}
M.~Zhang, \emph{Hodge filtration and {H}odge ideals for {$\mathbb{Q}$}-divisors with weighted homogeneous isolated singularities}, preprint arXiv:1810.06656.

\end{thebibliography}
\bibliographystyle{amsalpha}

\end{document}